\documentclass [12pt,a4paper,reqno]{amsart}
\usepackage{amsmath, amssymb}

\input xy
\xyoption{all}

\newcommand{\etype}[1]{\renewcommand{\labelenumi}{(#1{enumi})}}

\def\ealph{\etype{\alph}}
\def\dispace{\setlength{\itemsep}{2pt}}
\newcommand{\ds}[1]{\ {#1} \ }
\newcommand{\dss}[1]{\quad {#1} \quad }

\newcommand{\thminref}[2]{\pSkip\textbf{Theorem {#1}.}{ \emph{#2}}}

\def\tlT{\widetilde{T}}
\def\tlS{\widetilde{S}}
\def\tlM{\widetilde{M}}
\def\tlN{\widetilde{N}}
\def\tlD{\widetilde{D}}

\def\Dir{\dss \Rightarrow}

\def\noi{\noindent}

\def\pSkip{\vskip 1.5mm \noindent}
\def\semiring0{semiring$^{\dagger}$}

\newcommand{\Id}{\operatorname{Id}}
\newcommand{\hal}{\operatorname{hal}}
\newcommand{\SA}{\operatorname{SA}}

\newtheorem{thm}{Theorem} [section]
\newtheorem*{thm*}{Theorem}
\newtheorem{cor}[thm]{Corollary}

\newtheorem{prop}[thm]{Proposition}
\newtheorem*{prop*}{Proposition}

\newtheorem*{claim*} {Claim}

\newtheorem*{theorem13.5'} {Theorem 13.5$'$}
\newtheorem{acknowledgment*}[thm] {Acknowledgment}
\newtheorem{example}[thm]{Example}
\newtheorem{examp}[thm]{Example}

  \newtheorem{remarks}[thm]{Remarks}
    \newtheorem*{remarks*} {Remarks}
 \newtheorem*{remark*}{Remark}
 \newtheorem{defn}[thm]{Definition}
\newtheorem*{defn*}{Definition}

\newtheorem{notation}[thm]{Notation}

\newtheorem*{notation*} {Notation}

\def\F{\mathbb{F}}
\def\B{\mathbb{B}}

 \renewcommand{\sectionmark}[1]{}

\begin{document}

\title[ Generation of summand absorbing submodules ]
{Generation of summand absorbing submodules 
}
\author[Z. Izhakian]{Zur Izhakian}
\address{Institute  of Mathematics,
 University of Aberdeen, AB24 3UE,
Aberdeen,  UK.}
    \email{zzur@abdn.ac.uk 
    }
\author[M. Knebusch]{Manfred Knebusch}
\address{Department of Mathematics,
NWF-I Mathematik, Universit\"at Regensburg 93040 Regensburg,
Germany} \email{manfred.knebusch@mathematik.uni-regensburg.de}
\author[L. Rowen]{Louis Rowen}
 \address{Department of Mathematics,
 Bar-Ilan University, 52900 Ramat-Gan, Israel}
 \email{rowen@math.biu.ac.il}

\subjclass[2010]{Primary   14T05, 16D70, 16Y60 ; Secondary 06F05,
06F25, 13C10, 14N05}
\date{\today}


\keywords{Semiring,  lacking zero sums, direct sum decomposition,
free (semi)module,
summand absorbing submodule,
halo, additive spine.}


\thanks{\noindent \underline{\hskip 3cm } \\ File name: \jobname}


\begin{abstract}
An $R$-module $V$ over a semiring $R$ lacks zero sums (LZS) if $ x +
y = 0 \; \Rightarrow \;  x = y = 0$. More generally, a
submodule $W$ of $V$  is ``summand absorbing'', if $ \forall \,
x, y \in V: \ x + y \in W \; \Rightarrow \; x \in W, \; y \in W. $
These relate to  tropical algebra and modules over idempotent
semirings, as well as modules
over semirings of sums of squares. In previous work, we have explored the lattice of summand
absorbing submodules of a given LZS module, especially those that
are finitely generated, in terms of the lattice-theoretic Krull
dimension. In this note we describe their explicit generation.
\end{abstract}

\maketitle
\setcounter{tocdepth}{1}

{ \small \tableofcontents}

\numberwithin{equation}{section}
\section{Introduction}

Semirings, initiated by Costa \cite{Cos} and exposed by Golan
\cite{golan92}, have played an increasing role recently due to
increased interest to tropical algebra which involves   the max-plus
algebra and related constructions. Another important example is the
set of positive elements in an ordered ring. Both of these examples
 lack zero sums (termed ``zero sum free'' in \cite{golan92}) in the following sense:
An $R$-module $V$ over a semiring $R$ \textbf{lacks zero sums}
(abbreviated \textbf{LZS}),~ if
\begin{equation}\label{eq:1.27}
  \forall \, x, y \in V : \  x + y = 0 \Dir  x = y = 0. \tag{LZS}
\end{equation}
\noindent
As noted in \cite[Proposition~1.8]{Dec}, any module over an
idempotent semiring is LZS, yielding a large assortment of examples.
Furthermore, by \cite[Examples~1.6]{Dec} being LZS is closed under
submodules, direct products, and modules $\operatorname{Fun} (S, V)$
of functions from a set $S$ to a module~$V$.
Thus, examples  include the max-plus algebra, function semirings,
polynomial semirings and Laurent polynomial semirings over
idempotent semirings, and the ``boolean semifield'' $\mathbb B = \{
-\infty, 0 \}$ (and thus subalgebras of algebras that are free
modules over $\B$).  This shows that our results pertain to
``$\F_1$-geometry.''

Continuing the theory  from \cite{Dec},  we called a submodule $W$
of $V$  \textbf{summand absorbing} (abbreviated \textbf{SA}) in
$V$, if
\begin{equation}\label{eq:1.1}
\forall \, x, y \in V: \  x + y \in W \Dir x \in W, \; y \in W;
\tag{SA}
\end{equation} we then call $W$ an \textbf{SA-submodule} of
$V$. An \textbf{SA-left ideal} of a semiring $R$ is an SA-submodule of~$R$.

SA-submodules were the main subject of investigation in \cite{IKR8},
largely because of their nice lattice-theoretic properties. The
objective of \cite{IKR8} was to continue to develop the theory of
SA-modules over semirings, along the  lattice-theoretic  lines of
classical module theory (especially the Noetherian theory), with the
goal of paralleling the Wedderburn-Miller-Remak-Krull-Schmidt and
Jordan-H$\operatorname{\ddot{o}}$lder theorems.

Given an $R$-module $V$ and a set $S$ of generators of $V$, we
detect a new set~ $T$ of generators of $V$, which is ``small'' in
some sense if $S$ is ``small'', and gives us sets of generators of
all SA-submodules $W$ of $V$ in a coherent way. Here is an instance.

\begin{defn}\label{def:2.1} A set of generators $T$ of $V$
is \textbf{SA-adapted}, if every SA-submodule $W$ of~$V$ is generated
by the set $W \cap T$.\end{defn}

We   obtain a reasonable SA-adapted set of generators $T$ from a
given set of generators~ $S$ by employing the so-called
\textbf{additive spine} $M$ of a module:

\begin{defn}
\label{Def:3.1} Assume that $S$ is a subset of $V$.
\begin{itemize}
\item[a)] The \textbf{halo} $\tlS $ of $S$ in $V$ is the set of all $v \in V$ such that there exist $\lambda, \mu \in R$ with $\lambda v \in S$ and $\mu \lambda v = v$.
\item[b)] $S$ is called an \textbf{additive spine} of the $R$-module $V$, if $V$ is additively generated by $\tlS $, which we denote as
$V = \sum\limits^{\infty} \tlS $.
\end{itemize}
\end{defn}
\noi Thus the additive spines of $R$ are  of $R$ considered as a
left $R$-module.

\pSkip

In the special case
that both~$M$ and $S$ are finite it will turn out that also ~$T$ is
finite, and so all SA-submodules $W$ of $V$ are generated by at most $\vert
T \vert$ elements.  \begin{example}\label{examp:9.12}
 If $V$ has a finite SA-adapted  set of generators, then~$V$ is SA-artinian (as defined in \cite[Definition~1.4]{IKR8}). \end{example}

%

One main general result:

\thminref{\ref{Thm:3.3}}{Assume that $S$ is an additive spine of an
$R$-module $V$. Then every SA-submodule $W$ of $V$ is generated by
$W \cap S$, and moreover $W \cap S$ is an additive spine of~$W$.}

\pSkip
 For semirings, this specializes in \S\ref{sec:2} to:
\thminref{\ref{thm:2.9}}{Assume that $S$ is a set of generators of
an $R$-module $V$, and $M$ is an additive spine of $R$. Then any
SA-submodule $W$ of $V$ is generated by the set $W \cap (MS)$.}
\pSkip Consequently, if $V$ is generated by $n$ elements, then every
SA-submodule of $V$ is generated by $mn$ elements, where $m = |M|$
is independent of $n$. An application to matrices is given in
Theorem~\ref{Cor:2.15}, and more generally to monoid semirings in
Theorem~\ref{Theorem:2.18}.
%

Section \ref{sec:4} focuses on finite generation in terms of finite additive spines.

\section{Generating SA-submodules by use of additive spines}\label{sec:2}

 Throughout this paper, $R$ is a semiring (with $1 = 1_R$), and $V$ is
a (left) module (sometimes called ``semimodule'') over $R$; i.e.,
$(V,+)$ is a monoid satisfying the familiar module axioms as well
as $r 0_V = 0_R x = 0_V$ for all $r\in R,$ $ x \in V.$  The zero
submodule $\{ 0_V \}$ is usually written as~$0.$

We first  state basic facts about additive spines of $R$,
and give basic examples.

\begin{defn}\label{def:2.3} (Special case of Definition \ref{Def:3.1}).
Given a subset $M$ of $R$.
\begin{enumerate}
 \item[a)] We define  the set
\[ \tlM : = \{ x \in R \ds  \vert \, \exists \, y, z \in R :  \ y x \in M, \; z y x = x \},\]
which we call the  \textbf{halo} of $M$ in $R$.

\item[b)] If the halo $\tlM$ additively generates $R$, i.e., $R = \sum\limits^{\infty} \tlM$, we call $M$ an \textbf{additive spine of}~ $R$.
    \end{enumerate}

\end{defn}

We state some facts about halos which are immediate consequences of Definition \ref{def:2.3}.a.

\begin{remarks}\label{rem:2.4} $ $
\begin{itemize}\dispace
\item[i)] $M \subset \tlM$ for any set $M \subset R$.
\item[ii)] If $M \subset N \subset R$, then $\tlM \subset \tlN $.
\item[iii)] If $(M_i \ds \vert i \in I)$ is a family of subsets of $R$, then
$$ \Bigl( \bigcup\limits_{i \in I} M_i \Bigr)^{\sim} = \bigcup\limits_{i \in I} \tlM_i. $$

\item[iv)] $\{ 0 _R\}^{\sim} = \{ 0_R \}$ and $(M \setminus \{ 0 _R\})^{\sim} = \tlM.$
\end{itemize}
\end{remarks}

Due to the last remark we may assume in any study of halos that $0_R \in M$ or $0_R \not\in M$, whatever is more convenient.

\pSkip

Here are  perhaps the most basic examples of halos deserving interest.

\begin{example}
\label{examp:2.5} Let $M = \{ 1_R \}$. Then $\tlM$ is the set of left invertible elements of $R$. Indeed, if $x \in \tlM$, then there exists $y \in R$ with $yx = 1_R$. Conversely, if $x$ is left-invertible there exists $y \in R$ with $y x = 1_R$, and so $xyx = x$, which proves that $x \in \tlM$.
\end{example}

\begin{example}
\label{examp:2.6} Let $M = \{ e \}$ with $e$ an idempotent of $R$.
If $x \in \tlM$, then there exist $y, z \in R$ with $yx = e$, $ze =
x$. It follows that $xe = x$, yielding the von Neumann condition
$xyx = x$, cf.  \cite{IJK2}. Conversely, if $yx = e$ and $xyx = x$, then clearly $x
\in \tlM$. This proves that
\[ \{ e \}^{\sim} = \{ x \in R \ds \vert \, \exists \ y \in R :
 \ y x = e, \; xyx = x \}. \]
\end{example}

Let $\Id (R)$ denote the set of all idempotents of $R$. Starting from Example \ref{examp:2.6}, we obtain the following fact.

\begin{prop}
  \label{prop:2.7} If $R$ is any semiring, then
\[ \Id (R)^{\sim} = \{ x \in R \ds \vert \, \exists \ y \in R : xyx = x \} . \]
\end{prop}

 \begin{proof}$\Id (R)^{\sim}$ is the union of the sets $\{ e \}^{\sim}$ with $e$ an idempotent of $R$ (cf. Remark \ref{rem:2.4}.iii).
 Thus it is clear from Example~\ref{examp:2.6} that for every $x \in \Id (R)^{\sim}$ there exists some $y \in R$ with $xyx = x$.
 Conversely, if $xyx = x$, then $yx \cdot yx = yx$, and so $e : = yx$ is an idempotent of $R$. Moreover $xe = x$, and so $x \in \{ e \}^{\sim}$.
\end{proof}
We state an immediate consequence of this proposition.

\begin{cor}
\label{cor:2.8}   For any subset $M$ of $R$ we have
\[ [ M \cap \Id (R)]^{\sim} = \{ x \in \tlM \ds \vert \, \exists \ y \in R : xyx = x \}, \]
and $\tlM$ is the \textbf{disjoint} union of this set and $[ M \setminus \Id (R)]^{\sim}$. \end{cor}
\noi
The set $[ M \cap \Id (R)]^{\sim}$ may be regarded as the ``easy part'' of the halo $\tlM$.

 \begin{notation}\label{notat:2.1} Given (nonempty) subsets $A, B$ of $R$, we denote the set of all products $ab$ with $a \in A$, $b \in B$ by $AB$ (or $A \cdot B$).
 Similarly, if $A \subset R$ and $X \subset V$, then $AX$ denotes the set of products $ax$ with $a \in A$, $x \in X$. Furthermore,  $\sum\limits^{\infty} A$ and $\sum\limits^{\infty} X$ denote the set of all finite sums of elements of $A$ in $R$ and of $X$ in $V$ respectively.
Admitting also the empty sum of elements of $A$ or $X$, we always
have $0_R \in \sum\limits^{\infty} A$, $0_V \in \sum\limits^{\infty}
X$.
If necessary, we write more precisely $\sum\limits^{\infty}_R A$ and $\sum\limits^{\infty}_V X$ instead of $\sum\limits^{\infty} A$ and $\sum\limits^{\infty} X$. 
\end{notation}
\noindent
In this notation, a set $S \subset V$ generates the $R$-module $V$,
if $V = \sum\limits^{\infty} RS$.

\bigskip We are ready for a central result in this note.
 $\SA (V)$ denotes the poset consisting of all SA-submodules of $V$.

\begin{thm}
\label{thm:2.9} Assume that $S$ is a set of generators of a (left) $R$-module $V$, and $M$ is an additive spine of $R$. Then any SA-submodule $W$ of $V$ is generated by the set $W \cap (MS)$.
\end{thm}
\begin{proof}
 Since $V = \sum\limits^{\infty} RS$ and $R = \sum\limits^{\infty} \tlM$, we have $V = \sum\limits^{\infty} \tlM S$.

\noi
Let $w \in W$, $w \ne 0$, be given. Then
\begin{equation}\label{eq:a.1}
w = \sum\limits_{i = 1}^n x_i s_i \tag{$A$}\end{equation} with $n \in
\mathbb{N}$, $s_i \in S$, $x_i \in \tlM$. Since $W$ is in $\SA(V)$,
it follows that
\begin{equation*}\label{eq:2}
x_i s_i \in W \qquad \mbox{for} \; 1 \leq i \leq n. \end{equation*}
Now choose  $y_i, z_i \in R$ such that $m_i : = y_i x_i \in M$ and $x_i = z_i m_i$. Then
\begin{equation}\label{eq:a.3}
y_i (x_i s_i) = m_i s_i \in W \cap (MS) \tag{B}\end{equation}
and \[ z_i m_i s_i = z_i y_i x_i s_i = x_i s_i. \]
From \eqref{eq:a.1} we obtain that
\begin{equation}\label{eq:a.4}
w = \sum\limits_{i = 1}^n z_i (m_i s_i). \tag{C}\end{equation}
We conclude from \eqref{eq:a.3} and \eqref{eq:a.4} that $W \cap (MS)$ generates $W$.
\end{proof}

\begin{cor}\label{cor:2.10} Assume that $R$ has a finite additive spine $M$ and $V$ has a finite set of generators $S$. Then every SA-submodule $W$ of $V$ is finitely generated, more precisely, generated by at most $\vert M \vert \cdot \vert S \vert$ elements (independent of the choice of $W$!). \end{cor}

\begin{thm}\label{thm:2.11} Assume that $V$ is a module over a semiring $R$ that is
additively generated by the set of its left invertible elements. Then every set of generators $S$ of $V$ is SA-adapted. \end{thm}

\begin{proof} We read off from Example \ref{examp:2.5} that $\{ 1_R \}$ is an additive spine of $R$. So by Theorem~\ref{thm:2.9} every SA-submodule $W$ of $V$ is generated by $W \cap S = W \cap (1_R S)$.  \end{proof}

We take a look at additive spines of matrix semirings.

\begin{example}
\label{examp:2.12} Assume that $C$ is a semiring that is additively
generated by $\{ 1_C \}$, $$C = \sum\limits^{\infty} \{ 1_C \}.$$ In
other terms, the unique homomorphism $\varphi : \mathbb{N}_0 \to C$
with $\varphi (1) = 1_C$ is surjective. Then the semiring
\[ R = M_n (C) = \sum\limits_{i, j = 1}^n C e_{ij} \]
of $(n \times n)$-matrices with entries in $C$, and $e_{ij}$ the usual matrix units, has the additive spine
\[ D : = \{ e_{11}, e_{22}, \dots , e_{nn} \}. \]
Indeed, for every $j \in \{ 1, \dots, n \}$
\[ \{ e_{jj} \}^{\sim} \supset \{ e_{ij} \ds \vert 1 \leq i \leq n \}, \]
since $e_{ji} e_{ij} = e_{jj}$, $e_{ij} e_{jj} = e_{ij}$, and so $\tlD  = \bigcup\limits_{j} \{ e_{jj} \}^{\sim}$ contains the set $E : = \{ e_{ij} \ds \vert 1 \leq i, j \leq n \}$ of all matrix units, which by the nature of $C$ generates $M_{n} (C)$ additively.
\end{example}
This example can be amplified to a theorem about additive spines in
arbitrary matrix rings $M_n (A)$ by use of a general principle to
``multiply'' additive spines, which runs as follows:

\begin{prop}
\label{Prop:2.13} Assume that $R_1$ and $R_2$ are subsemirings of a semiring $R$, such that~$R$ is additively generated by $R_1 R_2$, i.e., $R = \sum\limits^{\infty} R_1R_2$. Assume moreover  that the elements of~ $R_1$ commute with those of $R_2$. Assume finally that $M_i$ is an additive spine of $R_i$. Let~ $\tlM_i$ denote the halo of $M_i$ in $R_i$ $(i = 1, 2)$. Then $\tlM_1 \tlM_2$ is contained in the halo $(M_1 M_2)^{\sim}$ of $M_1 M_2$ in $R$, and $M_1 M_2$ is an additive spine of $R$.
\end{prop}

\begin{proof}
   Let $x_i \in \tlM_i$  $(i = 1,2)$ be given. We have elements $y_i, z_i$ of $R_i$ with $m_i : = y_i x_i \in M_i$ and $z_i m_i = x_i$. Now
\[ (y_1 y_2) (x_1 x_2) = (y_1 x_1) (y_2 x_2) = m_1 m_2 \]
and
\[ (z_1 z_2) (m_1 m_2) = (z_1 m_1) (z_2 m_2) = x_1 x_2 . \]
This proves that $x_1 x_2 \in (M_1 M_2)^{\sim}$. It follows that
\[ \Bigl( \sum\limits^{\infty}\tlM_1 \Bigr) \cdot \Bigl( \sum\limits^{\infty}\tlM_2 \Bigr) = R_1 R_2 \]
and then that
\[ R = \sum\limits^{\infty} R_1 R_2 = \sum\limits^{\infty} \tlM_1 \tlM_2. \] \vskip -6mm
\end{proof}
\begin{thm}\label{Thm:2.14} Assume that $R$ is the semiring of $(n \times n)$-matrices over a semiring $A$,~ so
\[ R : = M_n (A) = \sum\limits_{i, j = 1}^n A e_{ij} \]
with the usual matrix units $e_{ij}$. Let $N$ be an additive spine of $A$. Then the set $M : = \bigcup\limits_{i = 1}^n Ne_{ii}$, consisting of the diagonal matrices with entries in $N$, is an additive spine of $R$.\end{thm}

\begin{proof} Let $C$ denote the smallest subsemiring of $A$, $C = \{ n \cdot 1_A \ds \vert n \in \mathbb{N} \}$.  We have seen that $R_1 : = M_n (C)$ has the additive spine $D : = \{ e_{ii} \ds \vert 1 \leq i \leq n \}$ (Example \ref{examp:2.12}). Let $R_2 : = A \cdot 1_R$.
This is the subsemiring of $R$ consisting of all matrices $a I$
 with $a \in A$, where ~$I$ is the identity matrix.   It has the additive spine $\mathbb N \cdot 1_{R_2}$. Now $R = \sum R_1 R_2$, and the elements of $R_1$ commute with those of $R_2$. Thus,  by Proposition \ref{Prop:2.13}, $R$ has the additive spine $D \cdot (\mathbb N 1_{R_2}) = \bigcup\limits_{i = 1}^n \mathbb N e_{ii}$. \end{proof}

Recalling Theorem \ref{thm:2.9} we obtain

\begin{thm}
\label{Cor:2.15} Assume that $V$ is an $M_n (A)$-module, $A$ any
semiring, and $S$ a system of generators of $V$. Assume furthermore
that $N$ is an additive spine of $A$. Then any SA-submodule $W$ of
$M_n (A)$ is generated by the set
\[ W \cap \bigg( \bigcup\limits_{i = 1}^n N e_{ii} \bigg) = \bigcup\limits_{i = 1}^n W \cap (N e_{ii}). \]
If $N$ is finite, then $W$ can be generated by at most $n \cdot \vert N \vert$ elements. \end{thm}

The proof of Theorem~\ref{Thm:2.14} can be seen in a much wider
context, as we explain now.
\begin{defn}
\label{Def:2.16} Let $S = (S, \cdot)$ be a multiplicative monoid. We
call a subset~$T$ of $S$ a \textbf{monoid spine of} $S$,   if for any
$s \in S$ there exist $s_1, s_2 \in S$ such that $t : = s_1 s \in T$
and $s_2 t = s$. \end{defn}

Given any semiring $A$ and monoid $S = (S, \cdot)$, we denote, as usual, the \textbf{monoid-semiring} of $S$ over $A$ by $A[S]$,
which is the free $A$-module with base $S \setminus \{ 0 \}.$

\medskip
In the case that the monoid $S$ is \textbf{without zero}, i.e., $S$ does \textbf{not} contain an absorbing element~0, ($0 \cdot S = S \cdot 0 = 0$ for all $s \in S$), the elements $x$ of $R : = A [S]$ are the formal sums
\[ x = \sum\limits_{s \in S} a_s s, \]

\noi with coefficients $a_s \in A$ uniquely determined by $x$, only
finitely many non-zero. The multiplication is determined by the rule
$(as) \cdot (bt) = (ab) (st)$ for $a, b \in A$, $s, t \in S$.
Identifying $a = a \cdot 1_S$, $s = 1_A \cdot s$, we regard $A$ as a
subsemiring of $R$ and $S$ as a submonoid of $(R, \cdot)$.

If the monoid $S$ has a zero $0 = 0_S$, i.e., is pointed, we take for $R = A [S]$ the
free $A$-module with base $S \setminus \{ 0 \}$ and multiplication
rule $(as) \cdot (bt) = (ab) (st)$ if $st \ne 0 $, $(as) (bt) = 0$
otherwise. Now the nonzero elements of $R = A [S]$ are formal sums
$\sum\limits_{s \ne 0} a_s s$. We identify again $a = a \cdot 1_S$,
$s = 1_A \cdot s$ for $s \in S \setminus \{ 0 \}$, and now also $0 =
0_A$. Then again $A$ becomes a subsemiring of $R$ and $S$ a
submonoid of $(R, \cdot)$. We have $R = \sum\limits^{\infty} AS$ in
both cases.

\begin{example}
 \label{Examp:2.17} The matrix semiring $M_n (A)$ coincides with $A [S]$, where $S$ is the monoid $\{ e_{ij} \ds \vert 1 \leq i, j \leq n \} \cup \{ 0 \}$ with multiplication rule $e_{ij} e_{kl} = \delta_{jk} e_{il}$. Note that $S$ has the monoid spine $\{ e_{11}, \dots, e_{nn} \} \cup \{ 0 \}$.
\end{example}

\begin{thm}
\label{Theorem:2.18} Assume that $S$ is a multiplicative monoid
(with   or without zero) and ~ $T$ is a monoid spine of $S$. Assume
furthermore that $A$ is a semiring and $N$ is an additive spine of
$A$. Then $N \cdot T$ is an additive spine of $A[S]$.
\end{thm}

\begin{proof} Let $R : = A [S]$ and $R_1 : = C [S] \subset R$, with $C$ the image of the (unique) homomorphism $\mathbb{N}_0 \to A$.
It is obvious that $R_1 = \sum\limits^{\infty} S$ and that $S$ is contained in the halo $\tlT $ of $T$ in $R_1$. Thus~ $T$ is a monoid
spine of $R_1$.
(In fact it can be verified that $\tlT  = \tlS  = S$.) Let $R_2 : =
A \subset R$. Then $R = \sum\limits^{\infty} R_1 \cdot R_2$ and the
elements of $R_1$ commute with those of $R_2$. The assertion follows
from Proposition \ref{Prop:2.13}. \end{proof}

\section{Halos and additive spines in $R$-modules}\label{sec:3}
%

 In the following $V$ is again an $R$-module.

\begin{example}
  \label{Examp:3.2} If $S$ is a set of generators of the $R$-module $V$ and $M$ is an additive spine of~ $R$, then we know by Theorem~\ref{thm:2.9} that MS is an additive spine of $V$.
\end{example}

Theorem \ref{thm:2.9} generalizes as follows:

\begin{thm}\label{Thm:3.3} Assume that $S$ is an additive spine of an $R$-module $V$. Then every SA-submodule $W$ of $V$ is generated by $W \cap S$, and moreover $W \cap S$ is an additive spine of~$W$.\end{thm}

 \begin{proof}a) We first verify that $V$ itself is generated by $S$. Since $V$ is additively generated by~$\tlS $,  for given nonzero $v \in V$ we have
\begin{equation}\label{eq:b.1}
v = \sum\limits_{i = 1}^n v_i \tag{A}, \end{equation}
with $n \in \mathbb{N}$, $v_i \in \tlS $. There exist $\lambda_i, \mu_i \in R$ such that
\begin{equation}\label{eq:b.2}
s_i : = \lambda_i v_i \in S, \tag{B}\end{equation}
\begin{equation}\label{eq:b.3}
v_i = \mu_i s_i, \tag{C}\end{equation}
and so by \eqref{eq:b.1}
\begin{equation*}\label{eq:4}
v = \sum\limits_{i = 1}^n \mu_i s_i, \end{equation*}
and we are done.

\pSkip \noi b) If now $W$ is an SA-submodule of $V$, and the above
element $v$ lies in $W$, then in Equation~\eqref{eq:b.1} all
summands $v_i$ are in $W$, and so the $s_i$ from \eqref{eq:b.2}  are
in $W \cap S$. We conclude
 from~\eqref{eq:b.2}  and~\eqref{eq:b.3}  that all $v_i$ are in the halo $(W \cap S)^{\sim}$ of $W \cap S$ in $W$, and we infer from~\eqref{eq:b.1}  that $W$ is additively generated by $(W \cap S)^{\sim}$, i.e., $W \cap S$ is an additive spine of $W$. As proved in a) the set $W \cap S$ generates the $R$-module $W$. \end{proof}

We write down a chain of propositions which turn out to be useful in
working with halos and additive spines. For clarity we sometimes
denote the halo of a set $S$ in $V$ more elaborately by $\hal_V (S)$
instead of $\tlS $.

\begin{prop}
\label{Prop:3.4} If $S$ is a subset of an $R$-module $V$ and $W$ a submodule of $V$, then
\[ W \cap \hal_V (S) = \hal_W (W \cap S) = \hal_V (W \cap S). \]
\end{prop}\begin{proof} Let $v \in \hal_V (S)$ be given. We choose $\lambda, \mu \in R$ with $\lambda v = s \in S$ and $\mu s = v$. If now $v \in W$ then $\lambda v = s \in W \cap S$, and so $v \in \hal_W (W \cap S)$. This proves that
\begin{equation}\label{eq:1}
W \cap \hal_V (S) \subset \hal_W (W \cap S). \tag{A}\end{equation}

\noi
Trivially
\begin{equation}\label{eq:30}
\hal_W (W \cap S) \subset \hal_V (W \cap S). \tag{B}\end{equation}

\noi
If $v \in \hal_V (W \cap S)$, then there exist $\lambda, \mu \in R$ with $\lambda v = s \in W \cap S$ and $\mu s = v$. It follows that $v \in W \cap \hal_V (S)$. This proves
\begin{equation}\label{eq:3}
\hal_V (W \cap S) \subset W \cap \hal_V (S). \tag{C} \end{equation}
\eqref{eq:1}--\eqref{eq:3} together imply the assertion of the proposition. \end{proof}

In   case $S \subset W$ the proposition reads as follows:

\begin{cor}
\label{Cor:3.5} Let $S \subset V$. Then the halo of $S$ in any
submodule $W \supset S$ of $V$ coincides with the halo of $S$ in
$V$. \end{cor} \noi Thus in practice the notation $\hal_V (S)$
instead of $\tlS $ is rarely needed.

\begin{prop}\label{Prop:3.6} Let $(V_i \ds \vert i \in I)$ be a family of submodules of the $R$-module $V$ and assume that for every $i \in I$ there is given a set $S_i \subset V_i$.

\begin{enumerate}\dispace
\item[a)] Then
\[ \bigcup\limits_{i \in I} \hal_{V_i} (S_i) = \hal_V \bigg( \bigcup\limits_{i \in I} S_i \bigg). \]

\item[b)] If $\sum\limits_{i \in I} V_i = V$ and each $S_i$ is an additive spine of $V_i$, then $\bigcup\limits_{i \in I} S_i$ is an additive spine of~$V$.
    \end{enumerate}
\end{prop}

\begin{proof} Let $S : = \bigcup\limits_{i \in I} S_i$.

\noi
a): We have $\hal_V (S) = \bigcup\limits_{i \in I} \hal_V (S_i)$ in complete analogy to Remark \ref{rem:2.4}.iii. Furthermore  $\hal_V (S_i) = \hal_{V_i} (S_i)$ by Corollary \ref{Cor:3.5}.

\noi
b): Let $\tlS _i : = \hal_{V_i} (S_i)$. Then $\bigcup\limits_{i \in I} \tlS _i = \tlS $, $\sum\limits^{\infty} \tlS _i = V_i$, and so
\[ \sum\limits^{\infty} \tlS  = \sum\limits_{i \in I} \bigg( \sum\limits^{\infty} \tlS _i \bigg) = \sum\limits_{i \in I} V_i = V. \] \vskip -6mm
\end{proof}

We now have a good hold on all additive spines of a free $R$-module
as follows:

\begin{prop}\label{Prop:3.7} Assume that $V$ is a free $R$-module with base $(v_i \ds \vert i \in I)$. Then every additive spine $S$ of $V$ has the form
\[ S = \bigcup\limits_{i \in I} M_i v_i, \]
where  every $M_i$ is an additive spine of $R$, as defined in \S\ref{sec:2}.\end{prop}

\begin{proof} We have $V = \bigoplus\limits_{i \in I} V_i$ with $V_i = R v_i \cong\; R$. The claim follows from Proposition \ref{Prop:3.6}. \end{proof}

\begin{prop}[Functoriality of halos and additive spines]
\label{Prop:3.8} Let $\varphi : V \to V'$ be an $R$-linear map between $R$-modules.
\begin{itemize}
\item[a)] If $S$ is a subset of $V$, then
\[ \varphi (\tlS ) \subset \varphi (S)^{\sim}. \]

\item[b)] If $S$ is an additive spine of $V$, then the $R$-module $\varphi (V)$ is additively generated by $\varphi (\tlS )$, and so $\varphi (S)$ is an additive spine of $\varphi (V)$.
\end{itemize}
\end{prop}
 \begin{proof}
 a): Let $x \in \tlS $. We have $\lambda, \mu \in R$ with $\lambda x = s \in S$, $\mu s = x$. It follows that $\lambda \varphi (x) = \varphi (s)$, $\mu \varphi (s) = \varphi (x)$, whence $\varphi (x) \in \varphi (S)^{\sim}$.

\pSkip
\noi
b): By Corollary \ref{Cor:3.5} we may replace $V$ by $\varphi (V)$, and so assume that $\varphi$ is surjective. We have $\sum\limits^{\infty} \tlS  = V$. Applying $\varphi$, we obtain
\[ \sum\limits^{\infty} \varphi (\tlS ) = \varphi (V). \]
It follows by a) that $\sum\limits^{\infty} \varphi (S)^{\sim} = \varphi (V)$. \end{proof}
\begin{cor}
\label{Cor:3.9} Assume that $R$ and $T$ are semirings and $V$ is an $(R,
T)$-bimodule, i.e., $V$ is a left $R$-module, a right $T$-module,
and
\[ \forall \, \lambda \in R, \mu \in T, v \in V : \ (\lambda v) \mu = \lambda (v \mu). \]
Let $S$ be a subset of $V$. As before let $\tlS $ denote the halo of
$S$ in $_R V$ (which means $V$ as a left $R$-module). Then, for any
$t \in T$
\[ \tlS  t \subset (St)^{\sim}. \]
If $S$ is an additive spine of $V$, then $\tlS t$ generates the left $R$-module $Vt$ additively, and so $St$ is an additive spine of $Vt$.
\end{cor}

\begin{proof} Apply Proposition \ref{Prop:3.8} to the endomorphism $v \mapsto vt$ of $_R V$. \end{proof}

\begin{cor}\label{Cor:3.10} If again $V$ is an $(R, T)$-bimodule and $t$ is a unit of $T$, then $\tlS  t = (St)^{\sim}$, and ~$S$ is an additive spine of $V$ iff $St$ is an additive spine of $V$.\end{cor}
\begin{proof} Let $u : = t^{-1}$. Then by Corollary \ref{Cor:3.9} $(\tlS  t) u \subset (St)^{\sim} u \subset (Stu)^{\sim} = \tlS $.
Multiplying by ~$t$, we obtain $\tlS  t \subset (St)^{\sim} \subset \tlS t$, whence $\tlS  t = (St)^{\sim}$, and then
\[ \sum\limits^{\infty} (St)^{\sim} = \bigg( \sum\limits^{\infty} \tlS  \bigg) t. \] \vskip -6mm
\end{proof}

\begin{examp}\label{Exam:3.11} $R$ is an $(R, R)$-bimodule in the obvious way. Thus, if $M$ is an additive spine of $R$
 and if $u$ is a unit of $R$, then $M u$ is again an additive spine of $R$. \end{examp}

\begin{examp}\label{Exam:3.12} Assume that $C$ is a semiring that is a homomorphic image of $\mathbb{N}_0$, and $R : = M_n (C)$. We have seen in Example~2.12 that $\{ e_{11}, \dots, e_{nn} \}$ is an additive spine of $R$. Let $\sigma \in S_n$.
Then $u : = \sum\limits_{i = 1}^n e_{i \sigma_{(i)}}$ is a unit of $R$, namely $u$ is the permutation matrix of $\sigma^{-1}$. We have $e_{ii} u = e_{i \sigma (i)}$, and conclude that $\{ e_{1 \sigma (1)}, \dots, e_{n \sigma (n)} \}$
is an additive spine of $M_n (C)$.
\end{examp}
We can generalize Proposition \ref{Prop:2.13} as follows:

\begin{prop}
\label{Prop:3.13} Assume that $R_1, R_2$ are commuting subsemirings
of a semiring $R$ with $R = \sum\limits^{\infty} R_1 R_2$, and that
$V_1, V_2$ are left modules over $R_1$ and $R_2$ respectively.
Assume furthermore that there is given a composition $V_1 \times V_2
\stackrel{\bullet}{\longrightarrow} V$ such that
\[ (\lambda_1 \lambda_2) (v_1 \bullet v_2) = (\lambda_1 v_1) \bullet (\lambda_2 v_2) \]
for any $\lambda_i \in R$, $v_i \in V_i \, (i = 1,2)$. Assume finally that $V = \sum\limits^{\infty} V_1 \bullet V_2$. Then, given subsets $S_i \subset V_i$ with halos $\tlS _i$ in the $R_i$-module $V_i \, (i = 1,2)$, the following holds.

\begin{enumerate}\ealph \dispace
 \item[a)] $\tlS _1 \bullet \tlS _2$ is contained in the halo $(S_1 \bullet S_2)^{\sim}$ of $S_1 \bullet S_2$ in $V$.

\item[b)]  If $S_i$ is an additive spine of $V_i \, (i = 1,2)$, then
\[ V = \sum\limits^{\infty} \tlS _1 \bullet \tlS _2 \]
and $S_1 \bullet S_2$ is an additive spine of $V$.
    \end{enumerate}

\end{prop}

\begin{proof} Let $v_i \in \tlS _i \, (i = 1,2)$. We have $\lambda_i, \mu_i \in R_i$ with $\lambda_i v_i = s_i \in S_i$, $\mu_i s_i = v_i$. Now
\[ (\lambda_1 \lambda_2) (v_1 \bullet v_2) = (\lambda_1 v_1) \bullet (\lambda_2 v_2) = s_1 \bullet s_2 \]

\noi
and $(\mu_1 \mu_2) (s_1 \bullet s_2) = (\mu_1 s_1) \bullet (\mu_2 s_2) = v_1 \bullet v_2$. This proves that $\tlS _1 \bullet \tlS _2 \subset (S_1 \bullet S_2)^{\sim}$. If now $\sum\limits^{\infty} \tlS _i = V_i \, (i = 1,2)$, then
\[ \sum\limits^{\infty} (\tlS _1 \bullet \tlS _2) \supset \bigg( \sum\limits^{\infty} \tlS _1 \bigg) \bullet \bigg( \sum\limits^{\infty} \tlS _2 \bigg) = V_1 \bullet V_2, \]

\noi
and so $\sum\limits^{\infty} (\tlS _1 \bullet \tlS _2) \supset \sum\limits^{\infty} V_1 \bullet V_2 = V$, whence $\sum\limits^{\infty} \tlS _1 \bullet \tlS _2 = V$. A fortiori $\sum\limits^{\infty} (S_1 \bullet S_2)^{\sim} = V$. \end{proof}

Note that Proposition \ref{Prop:2.13} is indeed a special case of
this proposition:  Given an $R$-module~ $V$, take $R_1 = R_2 = R$,
$V_1 = R$, $V_2 = V$ and the scalar product $R \times V \to V$.

\section{The posets $\SA (V)$, $\Sigma \SA (V)$ and $\Sigma_f \SA_f$ in good cases}\label{sec:4}

Assume  that $R$ has a finite additive spine $M$ consisting of $m := \vert M \vert$ elements. We have seen in \S\ref{sec:2} that, when $S$ is a set of generators of $V$, then every $W \in \SA (V)$ is generated by the set $W \cap (MS)$. Thus, if $s : = \vert S \vert$ is finite, we see
that the lattice $\SA (V)$ is finite, consisting of at most $2^{m|S|}$ elements. More generally we have the following fact.

\begin{thm}\label{Thm:4.1} Assume that $V_0$ is a submodule of an $R$-module $V$ and $S$ is a subset of $V$, such that $V$ is generated over $V_0$ by $S$, i.e.,
\begin{equation}\label{eq:4.1}
  V = V_0 + \sum\limits^{\infty} RS.
\end{equation}
Assume that $R$ has a finite additive spine $M$ consisting of $m :=
\vert M \vert$ elements.

Let $W_0 \in \SA (V)$ be given with $W_0 \subset V_0$, and consider the set
\begin{equation}\label{eq:4.2}
\SA (V;W_0, V_0): = \{ W \in \SA (V) \ds \vert W \cap V_0 = W_0 \}.
\end{equation}
Then, if $s := \vert S \vert$ is finite, this set $\SA (V;W_0, V_0)$
consists of at most $2^{ms}$ elements. Furthermore, any chain $W_0
\subsetneqq W_1 \subsetneqq \dots \subsetneqq W_r$ in $\SA (V;W_0,
V_0)$ has length $r \leq ms$.
\end{thm}

\begin{proof}
 Let $U$ denote the submodule of $V$ generated by $S$. We have $V = V_0 + U$. If $W \in \SA (V;W_0, V_0)$, then by (1.1)
\begin{equation}\label{eq:4.3}W = W \cap V_0 + W \cap U = W_0 + W \cap U, \end{equation}
and, of course, $W \cap U \in \SA (U)$. Since $\vert \SA (U) \vert
\leq 2^{ms}$, as stated above, we infer that $\vert \SA (V;W_0, V_0)
\vert \leq 2^{ms}.$ Also, if $W_0 \subsetneqq W_1 \subsetneqq \dots
\subsetneqq W_r$ is a chain in $\SA (V;W_0, V_0)$, we conclude from
\eqref{eq:4.3} for $U_i : = W_i \cap U$ that
\[ U_0 \subsetneqq U_1 \subsetneqq \dots \subsetneqq U_r. \]
Every $U_i$ is generated by the set $U_i \cap (MS)$ and so
\[ U_0 \cap (MS) \subsetneqq U_1 \cap (MS) \subsetneqq \dots \subsetneqq U_r \cap (MS). \]
This implies that $r \leq \vert MS \vert = ms$.
\end{proof}

We return to an arbitrary semiring $R$ and permit  infinite sums of SA-submodules,
writing $U \in \Sigma \SA(V) $,  called a \textbf{$\Sigma\SA$-submodule} of $V$.

\begin{thm}\label{Thm:4.4} Assume that $T$ is an additive spine of the $R$-module $V$ (cf. Definition~ \ref{Def:3.1}).
\begin{itemize}\dispace
\item[a)] Then any $U \in \Sigma \SA (V)$ is generated by the set $U \cap T$.
\item[b)]  $\vert T \vert = t$ is finite, then $\vert \Sigma \SA (V) \vert \leq 2^t$, and any chain
\[ U_0 \subsetneqq U_1 \subsetneqq \dots \subsetneqq U_r \]
in $\Sigma \SA (V)$ has length $r \leq t$.
\end{itemize}
\end{thm}
\begin{proof} a): Write $U = \sum_{i \in I} W_i$
with $W_i \in \SA (V)$. We know by Theorem \ref{Thm:3.3} that every
$W_i$ is generated by $W_i \cap T$. Thus $U$ is generated by the set
    \[ \bigcup\limits_{i \in I} (W_i \cap T) = \bigg( \bigcup\limits_{i \in I} W_i \bigg) \cap T. \]
A fortiori $U$ is generated by $U \cap T$. \pSkip b): Every $U \in
\Sigma \SA (V)$ is generated by the set $U \cap T \subset T$. We
have at most $2^t$ possibilities for this set, and so $\vert \Sigma
\SA (V) \vert \leq 2^t$. Furthermore, if $U_0 \subsetneqq \dots
\subsetneqq U_r$ is a chain in $\Sigma \SA (V)$, then
    \[U_0 \cap T \subsetneqq U_1 \cap T \subsetneqq \dots \subsetneqq U_r \cap T, \]
since each $U_i$ generated by $U_i \cap T$, and so $r \leq t$.
\end{proof}


We denote the set of all finitely generated SA-submodules of $V$ by $\SA_f(V)$, and the set of submodules of $V$, which are sums of finitely many elements of $\SA_f(V)$, by $\Sigma_f \SA_f(V)$. (Note that a module in $\Sigma_f \SA_f(V)$ is finitely generated, but perhaps not SA in $V$.)

\medskip
By a variation of our previous arguments we obtain

\begin{thm}
\label{thm:4.6} Assume that $R$ has a finite additive spine $M$, and
also that $U \in \Sigma_f \SA_f (V)$. Let $S$ be a finite set of
generators of $U$. Then every $W \in \SA_f (U)$ is generated by the
finite set $W \cap (MS)$ and every chain
\[ W \supsetneqq W_1 \supsetneqq W_2 \supsetneqq \ds \dots \supsetneqq W_r \]
in $\SA (U)$, hence in $\SA_f (U)$, has length $r \leq \vert M \vert \cdot \vert S \vert$. A fortiori this holds if $W$ and all $W_i$ are in $\SA_f (V)$.
\end{thm}

\begin{proof}
 Every $W \in \SA_f (U)$ is generated by the finite set $W \cap (MS)$, cf. Theorem \ref{thm:2.9}. Furthermore,  by the same theorem, every $W_i$ is generated by the subset $W_i \cap ~(MS)$ of $W \cap (MS)$. It follows that
\[ W \cap (MS) \ds \supsetneqq W_1 \cap (MS) \ds \supsetneqq \dots \ds \supsetneqq W_r \cap (MS) \]
and so $r \leq \vert W \cap (MS) \vert \leq \vert M \vert \cdot \vert S \vert$. It is obvious that every SA-submodule of $V$  contained in $U$ is SA in $U$.
\end{proof}


Our final result 
refers to modules with additive spines which are not necessarily finite.

\begin{thm}
\label{Thm:4.7} Assume that $T \subset V$ is an additive spine of the $R$-module $V$, and $U \in ~\Sigma \SA (V)$.
\begin{itemize}\dispace
\item[a)] Then $U$ is generated by the set $U \cap T$.
\item[b)] If $U$ is an $\SA_f$-sum in $V$, and $(W_i \ds \vert i \in I)$ is a family of finitely generated $\SA$-submodules of $V$ with $U = \sum\limits_{i \in I} W_i$, then every $W_i$ is generated by a finite subset $T_i$ of $W_i \cap T$, and so $U$ is generated by the subset $\bigcup\limits_{i \in I} T_i = T'$ of $T$. This subset $T'$ is an additive spine of $U$.
    \item[c)] If $U \in \Sigma_f \SA_f (V)$ then $U$ is generated by a finite subset of $U \cap T$, and this is an additive spine of $U$.
\end{itemize}
\end{thm}

\begin{proof} We choose a family $(W_i \ds \vert i \in I)$ in $\SA_f (V)$ with $U = \sum\limits_{i \in I} W_i$.\\
a): Done before (Theorem \ref{Thm:4.4}).\pSkip
b): We assume now that all $W_i$ are finitely generated. Every $W_i$ is generated by $W_i \cap T$ (Theorem \ref{Thm:3.3}). It follows that $W_i$ is generated by a finite subset $T_i$ of $W_i \cap T$. Indeed, given generators $s_1, \dots, s_r$ of $W_i$ for $i$ fixed, write every $s_j$ as a linear combination of a finite subset $T_{ij}$ of $W_i \cap T$. Then $T_i := \bigcup\limits_{j = 1}^r T_{ij}$ does it. It follows by Theorem~\ref{Thm:3.3} that $T_i$ is an additive spine of $W_i$. It now is clear that $T': = \bigcup\limits_{i \in I} T_i$ generates $U = \sum\limits_{i \in I} W_i$, and it follows by Proposition \ref{Prop:3.6} that $T'$ is an additive spine of $U$. \pSkip
c): Now evident, since the index set $I$ can be assumed to be finite, and so $T' = \bigcup\limits_{i \in I} T_i$ is a finite additive spine of $U$. \end{proof}


\begin{thebibliography}{IMS}


%



%


\bibitem{Cos}
A.A.~Costa.
\newblock Sur la  th\^{e}orie g\'{e}n\'{e}rale des
demi-anneaux,
\newblock {\em Publ. Math. Decebren} 10:14--29, 1963.

%



\bibitem{golan92}
J.~Golan.
\newblock {\em Semirings and their Applications}, Springer-Science + Business, Dordrecht, 1999.
\newblock (Originally published by Kluwer Acad. Publ.,  1999.)




%

\bibitem{IJK2} Z.~Izhakian, M.~Johnson, and M. Kambites.
Pure dimension and projectivity of tropical convex sets,  {\em Adv. in Math.}, 303:1236--1263, 2016.

%


%


%
%







\bibitem{Dec}
Z.~Izhakian, M.~Knebusch, and L.~Rowen.
\newblock Decompositions of modules lacking zero sums.
\newblock  \emph{Israel J. Math.},
 \newblock 225(2):503--524, 2018.

 \bibitem{IKR8}
Z.~Izhakian, M.~Knebusch, and L.~Rowen. \newblock Summand absorbing
submodules, \emph{J. Pure and Appl.~Alg.}, to appear.

%








%


%
%



%
%

\end{thebibliography}
\end{document}